\documentclass{tran-l}

\usepackage{amssymb}

\usepackage{enumerate}
\usepackage{tikz-cd}

\newtheorem{theorem}{Theorem}[section]
\newtheorem{lemma}[theorem]{Lemma}

\theoremstyle{definition}
\newtheorem{definition}[theorem]{Definition}
\newtheorem{proposition}[theorem]{Proposition}
\newtheorem{corolary}[theorem]{Corollary}
\newtheorem{example}[theorem]{Example}

\theoremstyle{remark}
\newtheorem{remark}[theorem]{Remark}

\numberwithin{equation}{section}

\newcommand{\restr}[2]{{\left.\mkern-1mu #1\right|}_{#2}}

\begin{document}

\title{On the stability and shadowing of tree-shifts of finite type}

\author{Dawid Bucki}
\address{} 
\email{dawid.bucki@doctoral.uj.edu.pl} 
\thanks{}


\date{}

\dedicatory{}

\begin{abstract}
We investigate relations between the pseudo-orbit-tracing property, topological stability and openness for tree-shifts.
We prove that a tree-shift is of finite type if and only if it has the pseudo-orbit-tracing property which implies that the tree-shift is topologically stable and all shift maps are open. 
We also present an example of a tree-shift for which all shift maps are open but which is not of finite type.
It also turns out, that if a topologically stable tree-shift does not have isolated points then it is of finite type. 
\end{abstract}

\maketitle

\section{Introduction}
Tree-shifts were introduced and studied in \cite{beal_1,beal_2,beal_3,beal_4,beal_5} as a step-in between one-dimensional shifts and multi-dimensional shifts defined over $\mathbb{Z}^d$. 
They are less complicated than multi-dimensional shifts but still enrich the dynamics of the classical one-dimensional case.
At the same time they preserve a lot of important properties of one-dimensional shifts, which often fail for the general $\mathbb{Z}^d$-case. 
The dynamics of tree-shifts is further studied in \cite{Ban:mixing} and \cite{Ban:mix_chaos} from which we derive our basics notions and terminology.\\
\indent Let $\Sigma$ be some finite non-empty set. 
A tree is a function assigning a label from some finite alphabet $\mathcal{A}$ to every node of the infinite $|\Sigma|$-ary tree which can be indexed by words over the alphabet $\Sigma$. 
Without loss of generality, one may consider only labeled binary trees, as all proofs easily generalise from the binary case to the case of $\Sigma$ with more elements. 
The space of all trees can be endowed with a metric compatible with the product topology on $\mathcal{A}^{\Sigma^*}$, where $\Sigma^*$ is the set of all words over $\Sigma$ endowed with the discrete topology.  
There are naturally arising shift maps denoted by $\sigma^i$, where the index $i$ is in $\Sigma$. 
The $i$-th map applied to a tree, gives us a sub-tree rooted at $i$-th node and relabels the nodes accordingly. 
A tree-shift is then defined as a closed subset of the family of all labeled trees invariant with respect to all shift actions $\sigma^i$, $(i\in\Sigma)$.\\
\indent It turns out that tree-shifts can be characterised by the forbidden sets of blocks, where a block is a labeled finite tree.
This is analogous to the way one-dimensional shifts are defined. 
For every tree-shift there is some set of blocks such that a labeled tree is in the tree-shift if and only if none of the forbidden blocks is a sub-block of the tree. 
If for a tree-shift $X$ this forbidden set defining $X$ can be finite then we say that $X$ is a tree-shift of finite type.\\
\indent The pseudo-orbit-tracing property (sometimes also called the shadowing property)  and topological stability were introduced in \cite{Bowen} and \cite{Walters1978} respectively.
They play a key role in the theory of stability of dynamical systems.
The pseudo-orbit-tracing property was also considered for actions of other groups and semi-groups in \cite{Kulczycki_Kwietniak} and in \cite{Osipov_Tikhomirov}.
In \cite{Walters1978} Walters proved that one-dimensional shifts are of finite type if and only if they have the pseudo-orbit-tracing property. In \cite{Oprocha2008ShadowingIM}, Oprocha showed that the same is true for multi-dimensional shifts over $\mathbb{Z}^d$. 
In section \ref{potp} we generalise Walter's proof obtaining similar result for tree-shifts. 
In section \ref{stab} we define topological stability for tree-shifts in an analogous way to \cite{Walters1978} and prove that if a tree-shift is of finite type, then it is topologically stable. 
We prove that topologically stable tree-shifts without isolated points have the pseudo-orbit-tracing property, therefore being of finite type.
This result was proved for zero-dimensional spaces equipped with a single transformation in \cite{stability} by Kawaguchi.\\
\indent In section \ref{open} we consider a generalisation of the notion of open shifts. 
This was introduced for one-dimensional shift spaces by Parry in \cite{parry}, where it is proved that the shift is open if and only if it is of finite type.
We prove that for tree-shifts, being of finite type implies that all shift maps are open and provide an example showing that the converse is false.
By doing so we provide another example showing that not all classical theorems concerning one-dimensional shifts of finite type can be generalised to tree-shifts (for another example of this type, see Theorem 3.3 in \cite{Ban:mixing}).
To the best of the author's knowledge there are no results characterising openness of shifts over semigroups.
Most of the multi-dimensional symbolic dynamics concerns actions of countable groups, where by definition every transformation in the action is a homeomorphism hence it is open.\\

This paper arose from author's Bachelor's Thesis written under the
supervision of Dominik Kwietniak.
In the final stages of this research, the author was supported by
National Science Centre (NCN), Poland grant no. 2022/47/O/ST1/03299.
For the purpose of Open Access, the author has applied a CC-BY public
copyright licence to any Author Accepted Manuscript (AAM) version
arising from this submission.
The author would like to express his gratitude to D. Kwietniak for his advices, introducing the author to the problem and for general guidance in author's first years of mathematical education.

The author also wants to thank anonymous reviewers for their careful reading of the manuscript and their comments that improved the quality of this article.

\section{Preliminaries}
By $\mathbb{N}$ we understand the set of non-negative integers and $n$ always denotes an integer.
Fix some alphabet (non-empty finite set) $\Sigma$.
For $n\geq 0$, $\Sigma^n$ stands for the set of words of length $n$ over the alphabet $\Sigma$ with $\Sigma^0=\{\epsilon\}$, where $\epsilon$ is the empty word. 
By $\Sigma^*=\bigcup_{n\geq 0} \Sigma^n$ we denote the set of all words over the alphabet $\Sigma$ and $\Sigma^{<n}=\bigcup_{k=0}^{n-1} \Sigma^k$ denotes the set of words whose length is smaller than $n\in\mathbb{N}$. 
For the word $w$ over $\Sigma$ and the set $B\subseteq\Sigma^*$ we put $wB = \{wu\ |\ u\in B\}$, where $wu$ is the concatenation of the words $w$ and $u$.
For alphabets $\Sigma$ and $\mathcal{A}$, a labeled tree on $\Sigma$ over $\mathcal{A}$ is a function $t\colon\Sigma^*\to\mathcal{A}$. 
One can imagine the set $\Sigma^*$ as an unlabeled $|\Sigma|$-ary infinite tree whose nodes correspond to words in $\Sigma^*$ and $\epsilon$ is the root.
Then labeled tree $t$ assigns to every node $w\in\Sigma^*$ a letter $t_w\in\mathcal{A}$. 
On the space $\mathcal{T}$ of all labeled trees on $\Sigma$ over $\mathcal{A}$, we define a metric $d$, by putting for labeled trees $s$ and $t$,
\begin{equation}\label{metric} d(s,t) = \begin{cases} 2^{-n-1},&\text{if }n=\min\{|w|\ |\ s_w\neq t_w\}<\infty,\\ 0,&\text{otherwise},\end{cases} \tag{$*$}
\end{equation}
where $|w|$ is the length of the word $w$. 
Observe that the metric space $(\mathcal{T},d)$ is compact.\\
From now on, we will refer to labeled trees as trees for simplicity.

\begin{remark}\label{balls-TS-remark}
    Note that for trees $s,t\in\mathcal{T}$ and $n\in\mathbb{N}$ we have that $d(s,t)<2^{-n}$ if and only if $\restr{s}{\Sigma^{<n}}=\restr{t}{\Sigma^{<n}}$.
\end{remark}

For two words $u,w\in\Sigma^*$, we say that $u$ is a prefix of $w$ if there is some word $v\in\Sigma^*$, such that $uv=w$. 
We say that a set $C\subseteq\Sigma^*$ is prefix-closed if for any $w\in C$ and $u$ being a prefix of $w$, we have $u\in C$. 
For example $\Sigma^{<n}$ is prefix-closed for every $n\in\mathbb{N}$.
Finite prefix-closed sets $C\subseteq\Sigma^*$ correspond to finite subtrees of the full $|\Sigma|$-ary tree rooted at the vertex $\epsilon$.
Given a finite prefix-closed set $C\subseteq \Sigma^*$, a function $p\colon C\to\mathcal{A}$ is called a pattern on $C$ over $\mathcal{A}$. 
Patterns correspond to labelled finite subtrees. 
A pattern $p\colon C\to\mathcal{A}$ is a sub-pattern of a tree $t\in\mathcal{T}$ if there is a word $w\in\Sigma^*$ such that
$$ \restr{t}{wC} = p.$$
If there is no such word $w$ for a pattern $p$, then we say that the tree $t$ omits $p$.
Patterns with domain $\Sigma^{<n}$, for $n>0$, are called blocks of height $n$.

To deal with dynamics on trees, we introduce the shift maps $\sigma^i\colon\mathcal{T}\to\mathcal{T}$, for $i\in\Sigma$, by defining $\sigma^i(t)_w = t_{iw}$ for $t\in\mathcal{T}$.
To shorten the notation we define $\sigma^w = \sigma^{w_{n-1}}\circ \ldots\circ\sigma^{w_0}$ for $w=w_0\ldots w_{n-1}\in\Sigma^n$. Then $\sigma^w(t)_u = t_{wu}$ for every $u,v\in\Sigma^*$ and $t\in\mathcal{T}$.
We also agree that $\sigma^\epsilon = \text{id}$.

\begin{definition}\label{TS}
    A set $X\subseteq\mathcal{T}$ is called a tree-shift (on $\Sigma$ over $\mathcal{A}$) if $X$ is closed and for every $i\in\Sigma$ it holds $\sigma^i(X)\subseteq X$.
\end{definition}

We can observe that if $\Sigma$ is a single element, then Definition \ref{TS} yields a classical one-sided shift over the alphabet $\mathcal{A}$.\\

For a block $b$ of height $n$ and a tree-shift $X$ we define the cylinder of $b$ as
$$ [b] = \left\{ t\in X\ |\ \restr{t}{\Sigma^{<n}}=b\right\}.$$
Note that $[b]$ depends on the choice of $X$ but we supress that from our notation.
We also denote the set of all blocks of height $n\geq 0$ over $\mathcal{A}$, appearing in trees from a tree-shift $X$, as $\mathcal{B}_n(X)$ and put $$\mathcal{B}(X)=\bigcup_{n>0}\mathcal{B}_n(X).$$ 
The set $\left\{[b]\ |\ b\in\mathcal{B}(X)\right\}$ is then a base for the topology on $X$ given by the metric defined by \eqref{metric} as if $b\colon \Sigma^{<n}\to\mathcal{A}$ is a block of height $n>0$ with $b\in\mathcal{B}_n(X)$, then there is $t\in X$ such that $\restr{t}{\Sigma^{<n}} = b$ and $[b]$ is the open ball of radius $2^{-n}$ centered at $t$.\\

Given a family of patterns $\mathcal{F}$, we define $\mathcal{T}_\mathcal{F}$ to be the set of all trees omitting all patterns from $\mathcal{F}$. 
Then $\mathcal{T}_\mathcal{F}$ is a closed subset of $\mathcal{T}$ and for every $i\in\Sigma$, we have
$\sigma^i(\mathcal{T}_\mathcal{F})\subseteq \mathcal{T}_\mathcal{F}$.
Therefore the family $\mathcal{T}_\mathcal{F}$ is a tree-shift. 
Of course for every non-empty tree-shift $X$ we have 
$ X = \mathcal{T}_{\mathcal{B}(\mathcal{T})\setminus\mathcal{B}(X)}$,
so every tree-shift can be characterised by some family of forbidden patterns.
Note that the family of forbidden patterns defining $X$ is not unique (compare with \cite{lind_marcus}).

\begin{definition}
    A tree-shift $X$ on $\Sigma$ over $\mathcal{A}$ is of finite type if $X=\mathcal{T}_\mathcal{F}$ for some finite set of patterns $\mathcal{F}$ on $\Sigma$ over $\mathcal{A}$.
\end{definition}

For the tree-shift of finite type we can always assume that its forbidden set of patterns $\mathcal{F}$ consists of blocks of some common height $n>0$, because it is always possible to extend forbidden patterns to forbidden patterns of arbitrary length. From now on we will be using this assumption.
For further reference we state a simple characterisation of tree-shifts of finite type.

\begin{lemma}\label{bl}
        A tree-shift $X$ is a tree-shift of finite type if and only if there is some $n>0$, such that for $t\in\mathcal{T}$, we have that $t\in X$ if and only if $t|_{w\Sigma^{<n}}\in\mathcal{B}_n(X)$ for every $w\in\Sigma^*$.
\end{lemma}
     
\section{Pseudo-orbit-tracing property}\label{potp} 
In this section we generalise Theorem 1 from \cite{Walters1978} to tree-shifts, adapting the approach from \cite{Walters1978} to our setting.
Let $X$ be a tree-shift on $\Sigma$ over $\mathcal{A}$.
\begin{definition}
        A $\delta$-pseudo-orbit in a tree-shift $X$ is a family $\left\{t^{(w)}\right\}_{w\in\Sigma^*}$ of trees from $X$ such that for every $w\in\Sigma^*$ and for any $i\in\Sigma$, we have 
        $$d\left(\sigma^i\left(t^{(w)}\right),t^{(wi)}\right)<\delta.$$
    \end{definition}
    
    If $\Sigma$ has a single element then the Definition \ref{potp} reduces to the usual definition of a $\delta$-pseudo-orbit.
    
    We can work with $\delta$ equal $2^{-n}$ for some $n\geq 0$, since we consider the metric $d$ on $X$ that can only attain such values. Therefore we will write ``$[n]$-pseudo-orbit'', instead of ``$2^{-n}$-pseudo-orbit''. 
    By Remark \ref{balls-TS-remark}, $\left\{t^{(w)}\right\}_{w\in\Sigma^*}$ is an $[n]$-pseudo-orbit if and only if for every $i\in\Sigma$ and $w\in\Sigma^*$ we have $$\restr{\sigma^i\left(t^{(w)}\right)}{\Sigma^{<n}} = \restr{t^{(w)}}{i\Sigma^{<n}}=\restr{t^{(wi)}}{\Sigma^{<n}}.$$
    
    \begin{lemma}\label{po}
        If $n>0$ and $\left\{t^{(w)}\right\}_{w\in\Sigma^*}$ is an $[n]$-pseudo-orbit in $X$, then for any $w,u,v\in\Sigma^*$, such that $uv\in\Sigma^{<n},$ we have $$t^{(w)}_{uv}=t^{(wu)}_v.$$ 
    \end{lemma}
    \begin{proof} 
        We will prove this lemma, by induction on the length of $u$.\\
        If $u=\epsilon$, then of course
        $$ t^{(w)}_{uv} = t^{(w)}_v = t^{(wu)}_v. $$
        Now, assume that $i\in\Sigma$ and $u\in\Sigma^*$ is such that $iuv\in\Sigma^{<n}$ and $$t^{(x)}_{uv}=t^{(xu)}_v$$
        for any $x\in\Sigma^*$. Then $$ t^{(wiu)}_v = t^{(wi)}_{uv}.$$
        But $\{t^{(w)}\}_{w\in\Sigma^*}$ is an $[n]$-pseudo-orbit and $uv\in\Sigma^{<n}$, so 
        $$ t^{(wiu)}_v = t^{(wi)}_{uv} = t^{(w)}_{iuv}.\qedhere $$
    \end{proof}
    
    \begin{definition}
        Let $T=\left\{t^{(w)}\right\}_{w\in\Sigma^*}$ be a family of trees in a tree-shift $X$, indexed by $\Sigma^*$.
        For $t\in X$ and $\varepsilon>0$, we say that $t$ $\varepsilon$-traces the family $T$, if for any $w\in\Sigma^*$, we have $$d\left(\sigma^w(t),t^{(w)}\right)<\varepsilon.$$
    \end{definition}
    
    Again, instead of writing ``$2^{-m}$-traces'' for some $m>0$, we will say ``$[m]$-traces''. 
    If a tree $t$ $[m]$-traces a family of trees $\left\{t^{(w)}\right\}_{w\in\Sigma^*}$, then by Remark \ref{balls-TS-remark} we have $$\restr{\sigma^w(t)}{\Sigma^{<m}}=\restr{t}{w\Sigma^{<m}}=\restr{t^{(w)}}{\Sigma^{<m}}.$$
    
    \begin{lemma}\label{trac}
        If trees $r,t\in X$ $[m]$-trace $\left\{t^{(w)}\right\}_{w\in\Sigma^*}\subseteq X$ for some $m>0$, then $r=t$.
    \end{lemma}
    \begin{proof} For every $w\in\Sigma^*$ we have
        $$ r_w = \sigma^w(r)_\epsilon = t^{(w)}_\epsilon = \sigma^w(t)_\epsilon = t_w.\qedhere $$
    \end{proof}
    
    \begin{definition}
        A tree-shift $X$ has the pseudo-orbit-tracing property, if for any $m> 0$, there is $n> 0$, such that every $[n]$-pseudo-orbit in $X$ is $[m]$-traced by some tree from $X$.
    \end{definition}
    
    It turns out that pseudo-orbit-tracing characterises the tree-shifts of finite type.
    
    \begin{theorem}\label{TSFT-POTP}
        A tree-shift $X$ is of finite type if and only if it has the pseudo-orbit-tracing property.
    \end{theorem}
    \begin{proof}
        We adapt the proof presented in \cite{Walters1978}.
        Assume that $X$ is of finite type and $X=\mathcal{T}_\mathcal{F}$ with $\mathcal{F}\subseteq\mathcal{B}_p(\mathcal{T})$ for some $p>0$. 
        Let $m> 0$. 
        We pick $n\geq\max\{p,m\}$. We will prove, that every $[n]$-pseudo-orbit in $X$ is $[m]$-traced by some point in $X$.\\
        Let $\left\{t^{(w)}\right\}_{w\in\Sigma^*}$ be an $[n]$-pseudo-orbit in $X$. We define $t\in\mathcal{T}$ by putting $t_w=t^{(w)}_\epsilon$ for $w\in\Sigma^*$.  
        By Lemma \ref{po} for $u\in\Sigma^{<m}$ and $w\in\Sigma^*$, we have
        $$ \sigma^w(t)_u = t_{wu} = t^{(wu)}_\epsilon = t^{(w)}_u. $$
        Therefore 
        $$\restr{\sigma^w(t)}{\Sigma^{<m}} = \restr{t^{(w)}}{\Sigma^{<m}}, $$ which means that $t$ indeed $[m]$-traces $\left\{t^{(w)}\right\}_{w\in\Sigma^*}$.\\
        
        \noindent It remains to prove that $t\in X$. 
        By the assumption that all forbidden words have length $p$, it suffices to check that
        $\restr{t}{w\Sigma^{<p}}\in\mathcal{B}_p(X)$, for any $w\in\Sigma^*$. 
        Fix $w\in\Sigma^*$.
        For $u\in\Sigma^{<p}$ we have $u\in\Sigma^{<n}$ because $n\geq p$ and
        $ t_{wu} = t^{(wu)}_\epsilon = t^{(w)}_u$, where the latter equality follows from Lemma \ref{po}. 
        It means that for any $w\in\Sigma^*$ we have
        $$\restr{t}{w\Sigma^{<p}} = \restr{t^{(w)}}{\Sigma^{<p}}$$ 
        and since we know that 
        $$\restr{t^{(w)}}{\Sigma^{<p}}\in\mathcal{B}_p(X),$$
        we conclude $\restr{t}{w\Sigma^{<p}}\in\mathcal{B}_p(X).$\\
        
        Now assume that $X$ has the pseudo-orbit-tracing property. Then there is some $n>1$ such that every $[n-1]$-pseudo-orbit is $[1]$-traced by some point from $X$. 
        We will prove that
        a tree $t$ is in $X$ if and only if for any $w\in\Sigma^*$, we have $\restr{t}{w\Sigma^{<n}}\in\mathcal{B}_n(X)$.
        Then Lemma \ref{bl} will yield that $X$ is of finite type. 
        The implication $(\Rightarrow)$ is by the definition, so let $t$ be such that for any $w\in\Sigma^*$, we have $\restr{t}{w\Sigma^{<n}}\in\mathcal{B}_n(X)$.
        This means that for any $w\in\Sigma^*$, there is some $t^{(w)}\in X$, such that $\restr{t^{(w)}}{\Sigma^{<n}} = \restr{t}{w\Sigma^{<n}}$. 
        Observe that the family $\left\{t^{(w)}\right\}_{w\in\Sigma^*}$ is an $[n-1]$-pseudo-orbit in $X$, because for $i\in\Sigma$ and $w\in\Sigma^*$, we have $$\restr{t^{(wi)}}{\Sigma^{<n-1}} = \restr{t}{wi\Sigma^{<n-1}} = \restr{t^{(w)}}{i\Sigma^{<n-1}} = \restr{\sigma^i\left(t^{(w)}\right)}{\Sigma^{<n-1}}.$$
        This $[n-1]$-pseudo-orbit is then $[1]$-traced by some $r\in X$, because of the pseudo-orbit-tracing property. 
        It means that for every $w\in\Sigma^*$ it holds $$r_w=\sigma^w(r)_\epsilon = t^{(w)}_\epsilon = t_w.$$ 
        So $r=t$ and thus $t\in X$.\qedhere
    \end{proof}

\section{Topological stability of tree-shifts}\label{stab}
    In the theory of one-dimensional shifts it is known, that if the shift has the pseudo-orbit-tracing property then it is topologically stable, which is proved in \cite{Walters1978}. 
    Also, if a one-dimensional shift is stable and has no isolated points then it has the pseudo-orbit-tracing property. 
    This can be concluded from \cite{stability}, where Kawaguchi proves this fact for general zero-dimensional spaces. 
    In this section we extend this result to tree-shifts.
    
    \begin{definition}\label{definition_stability}
        A tree-shift $X$ is topologically stable if for every $\varepsilon>0$ there exists $\delta>0$ such that if $\left\{\tau^i\right\}_{i\in\Sigma}$ is a family of continuous maps from $X$ to itself, satisfying
        $ d\left(\tau^i(t),\sigma^i(t)\right)<\delta$,
        for any $i\in\Sigma$ and $t\in X$,
        then there exists a continuous mapping $\varphi\colon X\to X$ such that for any $i\in\Sigma$ and $t\in X$ it holds
        $ d\left(\varphi(t),t\right)<\varepsilon$ and $\sigma^i\circ\varphi = \varphi\circ\tau^i$.
    \end{definition}
    
    When we consider the set $\{\tau^i\}_{i\in\Sigma}$ as above, then for $w\in\Sigma^n$, $\tau^w$ stands for $\tau^{w_{n-1}}\circ\ldots\circ\tau^{w_{0}}$.
    Again, to make our reasoning easier to follow, we will restate the above definition in an equivalent way.
    
    \begin{definition}\label{definition_nclose}
        If $X$ is a tree-shift, then maps $\zeta,\xi\colon X\to X$ are said to be $[n]$-close for $n> 0$, if for any $t\in X$ we have
        $$ \restr{\zeta(t)}{\Sigma^{<n}} = \restr{\xi(t)}{\Sigma^{<n}}.$$
        The families $\left\{\zeta^i\right\}_{i\in\Sigma}$, $\left\{\xi^i\right\}_{i\in\Sigma}$ of the maps from $X$ to $X$ are said to be $[n]$-close if $\zeta^i$ and $\xi^i$ are $[n]$-close for any $i\in\Sigma$.
    \end{definition}
    
    The following proposition follows immediately from Definitions \ref{definition_stability} and \ref{definition_nclose}.

    \begin{proposition}\label{stab-eq}
        A tree-shift $X$ is topologically stable if and only if for any $m>0$ there exists $n>0$ such that for any family $\left\{\tau^i\right\}_{i\in\Sigma}$ of continuous mappings from $X$ to $X$ that is $[n]$-close to the family $\left\{ \sigma^i\right\}_{i\in\Sigma}$ there is a continuous map $\varphi\colon X\to X$ that is $[m]$-close to identity and satisfies
        $ \varphi\circ\tau^i = \sigma^i\circ \varphi$
        for any $i\in\Sigma$.
    \end{proposition}   
    
    It is convenient to consider finite pseudo-orbits in a tree-shift.
    
    \begin{definition}
        Let $m,N>0$. A finite $[m]$-pseudo-orbit of order $N$ in a tree shift $X$ is a family $\left\{t^{(w)}\right\}_{w\in\Sigma^{<N}}$ of trees from $X$ satisfying
        $$ \restr{\sigma^i\left(t^{(w)}\right)}{\Sigma^{<m}} = \restr{t^{(wi)}}{\Sigma^{<m}}, $$
        for every $w\in\Sigma^{<N-1}$.
    \end{definition}
    
    For a finite pseudo-orbit we can define tracing and pseudo-orbit-tracing property in the same way as in the infinite case.
    
    \begin{definition}
        Let $n,N>0$.
        We say that a tree $t\in X$ $[n]$-traces a family $\{t^{(w)}\}_{w\in\Sigma^{<N}}$ of trees from $X$, if for any $w\in\Sigma^{<N}$ we have
        $$ \restr{\sigma^w(t)}{\Sigma^{<n}} = \restr{t^{(w)}}{\Sigma^{<n}}.$$
    \end{definition}
    
    \begin{definition}
        A tree-shift $X$ has the finite pseudo-orbit-tracing property if for any $m>0$ there is some $n>0$ such that for any $N>0$, any finite $[n]$-pseudo-orbit of order $N$ from $X$ is $[m]$-traced by a tree from $X$.
    \end{definition}
    
    Of course, the pseudo-orbit-tracing property implies the finite pseudo-orbit-tracing property. In compact spaces the converse is also true.
    
    \begin{proposition}
        If a tree-shift $X$ has the finite pseudo-orbit-tracing property then it has the pseudo-orbit-tracing property.
    \end{proposition}
    \begin{proof}
        Fix $m>0$ and pick $n>0$ using the finite pseudo-orbit-tracing property.
        Let $T=\left\{t^{(w)}\right\}_{w\in\Sigma^*}$ be an $[n]$-pseudo-orbit.
        For $N\in\mathbb{N}$, let $t^{N}$ be a tree that $[m]$-traces the finite pseudo-orbit $\left\{t^{(w)}\right\}_{w\in\Sigma^{<N}}$.
        Then, by compactness we can assume that the sequence $\left\{t^{N}\right\}_{N>0}$ converges to some $t\in X$ as $N\to\infty$.
        Let us prove that $t$ $[m]$-traces the $[n]$-pseudo-orbit $T$.
        Fix $w\in\Sigma^*$ and take $N>|w|$ big enough to have $$\restr{\sigma^w(t)}{\Sigma^{<m}} = \restr{\sigma^w\left(t^{N}\right)}{\Sigma^{<m}}.$$
        The tree $t^{N}$ $[m]$-traces the finite pseudo-orbit $\left\{t^{(v)}\right\}_{v\in\Sigma^{<N}}$ so $$\restr{\sigma^w\left(t^{N}\right)}{\Sigma^{<m}} = \restr{t^{(w)}}{\Sigma^{<m}}.$$
        Thus we obtain that 
        $$ \restr{\sigma^w(t)}{\Sigma^{<m}} = \restr{t^{(w)}}{\Sigma^{<m}}$$
        so the tree $t$ indeed $[m]$-traces the $[n]$-pseudo-orbit $T$.
    \end{proof}
    
    We will now prove that every tree-shift $X$ of finite type is topologically stable. As we already know, tree-shifts of finite type have the pseudo-orbit-tracing property. We can use this fact to obtain stability.
    
    \begin{theorem}\label{POTP-stab}
        If a tree-shift $X$ has the pseudo-orbit-tracing property, then it is topologically stable.
    \end{theorem}
    \begin{proof}
        We will use an equivalent condition for topological stability given by Proposition \ref{stab-eq}.\\
        Fix $m>0$. By the pseudo-orbit-tracing property there is some $n>0$ such that every $[n]$-pseudo-orbit in $X$ is $[m]$-traced by some tree from $X$. 
        Let $\left\{\tau^i\right\}_{i\in\Sigma}$ be a set of mappings such that for any $i\in\Sigma$ and $t\in X$ we have
        $$ \restr{\tau^i(t)}{\Sigma^{<n}} = \restr{\sigma^i(t)}{\Sigma^{<n}}. $$
        Let us now construct a map $\varphi\colon X\to X$ as follows.
        Observe that for every tree $t\in X$ the set $\left\{\tau^w(t)\right\}_{w\in\Sigma^*}$ is an $[n]$-pseudo-orbit in $X$ as for every $i\in\Sigma$ and $w\in\Sigma^*$ we have
        $$ \restr{\sigma^i\left(\tau^w(t)\right)}{\Sigma^{<n}} = \restr{\tau^i\left(\tau^w(t)\right)}{\Sigma^{<n}} = \restr{\tau^{wi}(t)}{\Sigma^{<n}}.$$
        This pseudo-orbit is then $[m]$-traced by some tree $\varphi(t)\in X$. 
        By Lemma \ref{trac}, a tree tracing a pseudo-orbit is unique so the map $t\mapsto\varphi(t)$ is well-defined. 
        We now need to prove that the map $\varphi$ is continuous. 
        If $\left\{t^k\right\}_{k=0}^\infty$ is a sequence in $X$ converging to some $t\in X$ and $\varphi\left(t^k\right)$ is a tree $[m]$-tracing the $[n]$-pseudo-orbit $\left\{\tau^w\left(t^k\right)\right\}_{w\in\Sigma^*}$ then we want to prove that $\varphi\left(t^k\right)\to \varphi(t)$ as $k\to\infty$. 
        Observe that $\varphi\left(t^k\right)$ is defined by
        \begin{equation}\label{E} \varphi\left(t^k\right)_w = \sigma^w\left(\varphi\left(t^k\right)\right)_\epsilon = \tau^w\left(t^k\right)_\epsilon.\end{equation}
        It follows that for every $k$ big enough (depending on $n$), $\varphi\left(t^k\right)_w$ is equal to $\tau^w(t)_\epsilon$ because $\tau^w$ is a continuous mapping and $t^k$ converges. 
        We also know that for every $w\in\Sigma^*$ it holds
        \begin{equation}\label{EE} \varphi(t)_w = \sigma^w\left(\varphi(t)\right)_\epsilon = \tau^w(t)_\epsilon.\end{equation}
        Equations \eqref{E} and \eqref{EE} imply that for $w\in\Sigma^*$ we have
        $$ \varphi\left(t^k\right)_w = \tau^w\left(t^k\right)_\epsilon \longrightarrow \tau^w(t)_\epsilon = \varphi(t)_w\text{ as }k\to\infty.$$
        It follows that $\varphi(t^k)\to\varphi(t)$ as $k\to\infty$.
        The condition of $\varphi$ being close to the identity is also satisfied as for every $t\in X$ we have that $\varphi(t)$ $[m]$-traces the pseudo-orbit $\left\{\tau^w(t)\right\}_{w\in\Sigma^*}$ so
        $$ \restr{\varphi(t)}{\Sigma^{<m}} = \restr{\sigma^\epsilon(\varphi(t))}{\Sigma^{<m}} = \restr{\tau^\epsilon(t)}{\Sigma^{<m}} = \restr{t}{\Sigma^{<m}}.$$
        Finally, for any $w\in\Sigma^*$ and $i\in\Sigma$, using \eqref{E}, we also have
        $$ \varphi\left(\tau^i(t)\right)_w = \sigma^w\left(\varphi\left(\tau^i(t)\right)\right)_\epsilon = \tau^w\left(\tau^i(t)\right)_\epsilon = \tau^{iw}(t)_\epsilon = \varphi(t)_{iw} = \sigma^i\left(\varphi(t)\right)_w,$$
        so
        $$ \sigma^i\circ\varphi = \varphi\circ\tau^i.$$
        We have constructed the mapping $\varphi$ satisfying the conditions from Proposition \ref{stab-eq}, thus the tree-shift $X$ is topologically stable.
    \end{proof}
    
    \begin{corolary}\label{TSFT-stab}
        Every tree-shift of finite type is topologically stable.
    \end{corolary}
    \begin{proof}
        If $X$ is a tree-shift of finite type then by Theorem \ref{TSFT-POTP}, $X$ has the pseudo-orbit-tracing property. 
        We conclude from Theorem \ref{POTP-stab} that it is topologically stable.
    \end{proof}

    \begin{theorem}
        If $X$ is a topologically stable tree-shift with no isolated points then it has the finite pseudo-orbit-tracing property.
    \end{theorem}
    \begin{proof}
        Let $m>0$ and take $n>m$ from the definition of topological stability and let $N>1$ (for $N\leq 1$ there is nothing to prove). We will prove that every finite $[n]$-pseudo-orbit of order $N$ in $X$ is $[m]$-traced by some point from $X$. 
        Fix $T=\left\{t^{(w)}\right\}_{w\in\Sigma^{<N}}$ such that for every $w\in\Sigma^{<N-1}$ and $i\in\Sigma$ we have
        $$ \restr{\sigma^i\left(t^{(w)}\right)}{\Sigma^{<n}} = \restr{t^{(wi)}}{\Sigma^{<n}}.$$
        The first step is to turn the map $w\mapsto t^{(w)}$ into an injection on $\Sigma^{<N}$. 
        Assume that there are some $u,v\in\Sigma^{<N}$ such that $u\neq v$ and
        $ t^{(u)} = t^{(v)}$ appear in $T$.
        We know that $t^{(v)}$ is not an isolated point in $X$ (as we assumed that $X$ has no isolated points), so there are infinitely many $t\in X$ such that 
        $$ t|_{\Sigma^{<n+1}} = t^{(v)}|_{\Sigma^{<n+1}}.$$
        We then choose $t$ that does not appear in $\left\{t^{(w)}\right\}_{w\in\Sigma^{<N}}$ and replace $t^{(v)}$ by $t$. 
        After finitely many steps we obtain the family $S=\left\{s^{(w)}\right\}_{w\in\Sigma^{<N}}$ such that 
        $w\mapsto s^{(w)}$ is an injection over $\Sigma^{<N}$ and
        for each $w\in\Sigma^{<N}$ we have
        $$ \restr{s^{(w)}}{\Sigma^{<n+1}} = \restr{t^{(w)}}{\Sigma^{<n+1}}.$$ 
        Observe that $S$ is still an $[n]$-pseudo-orbit of order $N$ as for $w\in\Sigma^{<N-1}$ we have
        $$ \restr{\sigma^i\left(s^{(w)}\right)}{\Sigma^{<n}} = \restr{s^{(w)}}{i\Sigma^{<n}} = \restr{t^{(w)}}{i\Sigma^{<n}} = \restr{\sigma^i(t^{(w)})}{\Sigma^{<n}} = \restr{t^{(wi)}}{\Sigma^{<n}} = \restr{s^{(wi)}}{\Sigma^{<n}}.$$ 
        
        The second step is to use $S$ to construct a family of maps $\left\{\tau^i\right\}_{i\in\Sigma}$. 
        Our family $S$ is finite and two trees in $S$ corresponding to different words are different, so we can find an integer $M>\max\{m,n\}$ such that for any $u,v\in\Sigma^{<N}$ with $u\neq v$ we have
        $$ \restr{s^{(u)}}{\Sigma^{<M}} \neq \restr{s^{(v)}}{\Sigma^{<M}}.$$
        For $i\in\Sigma$ and $t\in X$, we define $\tau^i(t)$ as follows.
        If there is $w\in\Sigma^{<N-1}$ such that
        $$ \restr{t}{\Sigma^{<M}} = \restr{s^{(w)}}{\Sigma^{<M}},$$
        then we put
        $ \tau^i(t) = s^{(wi)}.$
        Observe that if such a word exists then it is unique, because of the way we constructed $S$.
        Otherwise we put
        $ \tau^i(t) = \sigma^i(t).$
        The maps $\left\{\tau^i\right\}_{i\in\Sigma}$ are well defined and it is trivial to check that they are continuous. 
        Let us prove that $\tau^i$ is $[n]$-close to $\sigma^i$ for every $i\in\Sigma$. 
        If $t\in X$ and for every $w\in\Sigma^{<N-1}$ we have
        $$ \restr{t}{\Sigma^{<M}} \neq \restr{s^{(w)}}{\Sigma^{<M}},$$
        then of course
        $$ \restr{\tau^i(t)}{\Sigma^{<n}} = \restr{\sigma^i(t)}{\Sigma^{<n}}.$$
        If for some $w\in\Sigma^{<N-1}$ it holds
        $$ \restr{t}{\Sigma^{<M}} = \restr{s^{(w)}}{\Sigma^{<M}},$$
        then
        $$ \restr{\tau^i(t)}{\Sigma^{<n}} = \restr{s^{(wi)}}{\Sigma^{<n}} = \restr{\sigma^i(s^{(w)})}{\Sigma^{<n}} = \restr{s^{(w)}}{i\Sigma^{<n}} = \restr{t}{i\Sigma^{<n}} = \restr{\sigma^i(t)}{\Sigma^{<n}}.$$
        Hence, for each $i\in\Sigma$ the map $\tau^i$ is $[n]$-close to $\sigma^i$ and by topological stability we get, that there is some continuous $\varphi\colon X\to X$ such that 
        $$ \sigma^i\circ\varphi = \varphi\circ\tau^i $$
        for any $i\in\Sigma$, and $\varphi$ is $[m]$-close to the identity on $X$. We will prove, that $\varphi\left(t^{(\epsilon)}\right)$ $[m]$-traces the $N$-finite $[n]$-pseudo-orbit $\left\{t^{(w)}\right\}_{w\in\Sigma^{<n}}$. 
        Indeed, for $w\in\Sigma^{<N}$ we have
        $$ \restr{\sigma^w\left(\varphi\left(t^{(\epsilon)}\right)\right)}{\Sigma^{<m}} = \restr{\varphi\left(\tau^w\left(t^{(\epsilon)}\right)\right)}{\Sigma^{<m}} = \restr{\tau^w(t^{(\epsilon)})}{\Sigma^{<m}} =  \restr{s^{(w)}}{\Sigma^{<m}} = \restr{t^{(w)}}{\Sigma^{<m}}.$$
        The penultimate equality is by the definition of $\tau$ mappings and induction as 
        $$ \tau^w\left(t^{(\epsilon)}\right) = \tau^w\left(s^{(\epsilon)}\right) = \tau^{w_0w_1\ldots}\left(s^{(\epsilon)}\right) = \tau^{w_1w_2\ldots}\left(s^{(w_0)}\right) = \ldots = s^{(w)}.$$ 
        Therefore, the finite $[n]$-pseudo-orbit $\left\{ t^{(w)}\right\}_{w\in\Sigma^{<N}}$ is $[m]$-traced by some tree from $X$, so $X$ has the finite pseudo-orbit-tracing property.
    \end{proof}
    
    We note two immediate corollaries.
    
    \begin{corolary}
        If a tree-shift is topologically stable and has no isolated points then it is of finite type.
    \end{corolary}
    
    \begin{corolary}
        A tree-shift with no isolated points is topologically stable if and only if it is of finite type.
    \end{corolary}
    
\section{Open tree-shifts}\label{open}
    In \cite{parry} Parry proves that shift spaces such that the mapping $\restr{\sigma}{X}$ is open, are exactly shifts of finite type. It appears that this result cannot be fully generalised to the tree-shift case. In this section, we present an example of a tree-shift for which all shift maps are open, but which is not of finite type. However we prove that the converse implication holds: if a tree-shift $X$ is of finite type then for every $i\in\Sigma$ the map $\restr{\sigma^i}{X}$ is open.
    
    \begin{theorem}\label{TSFT-open}
        If a tree-shift $X$ is of finite type, then for every $i\in \Sigma$, the shift map $\restr{\sigma^i}{X}$ is open.
    \end{theorem}
    \begin{proof}
        We generalise the proof presented in \cite{parry}.
        Let $X$ be a tree-shift of finite type. We can assume that its forbidden set $\mathcal{F}$ is contained in $\mathcal{B}_n(\mathcal{T})$ for some $n>0$. It suffices to show that for any $b\in\mathcal{B}_m(X)$, where $m>n$, and $i\in\Sigma$, the set $\sigma^i\left([b]\right)$ is open, because sets of this form constitute a base for the topology of $X$. 
        Let $t\in\sigma^i\left([b]\right)$ and let $r\in [b]$ be such that $\sigma^i(r)=t$. We will prove that 
        $$\left[\restr{t}{\Sigma^{<m}}\right]\subseteq\sigma^i\left([b]\right).$$
        Let $s\in\left[\restr{t}{\Sigma^{<m}}\right]$ and define $q\colon\Sigma^*\to\mathcal{A}$ as follows. For $w\in\Sigma^*$, put $q_{iw}=s_w$ and for words $u\in\Sigma^*$, that do not start with $i$, put $q_w=r_w.$
        It is then obvious that $\sigma^i(q)=s$.
        The tree $q$ is an element of cylinder $[b]$ because 
        $$ q_\epsilon = r_\epsilon = b_\epsilon,$$
        $$ \restr{q}{i\Sigma^{<m-1}} = \restr{s}{\Sigma^{<m-1}} = \restr{t}{\Sigma^{<m-1}} = \restr{\sigma^i(r)}{\Sigma^{<m-1}} = \restr{r}{i\Sigma^{<m-1}} = \restr{b}{i\Sigma^{<m-1}},$$
        and for $j\in\Sigma\setminus\{i\}$ there is
        $$ \restr{q}{j\Sigma^{<m-1}} = \restr{r}{j\Sigma^{<m-1}} = \restr{b}{j\Sigma^{<m-1}}.$$
        Therefore
        $$ \restr{q}{\Sigma^{<m}} = \restr{b}{\Sigma^{<m}}.$$
        Also, for any $w\in\Sigma^*$, we have $\restr{q}{w\Sigma^{<n}}\in\mathcal{B}_n(X)$.
        If $w$ starts with $i$ then $w=iu$ for some word $u$ and
        $$ \restr{q}{iu\Sigma^{<n}} = \restr{s}{u\Sigma^{<n}}\in\mathcal{B}_n(X).$$
        If $w$ starts with some other letter then
        $$ \restr{q}{w\Sigma^{<n}} = \restr{r}{w\Sigma^{<n}}\in\mathcal{B}_n(X).$$
        Finally, if $w=\epsilon$ then 
        $$ \restr{q}{\Sigma^{<n}} = \restr{b}{\Sigma^{<n}}\in\mathcal{B}_n(X).$$
        Thus $q\in X$ and so we get that $s\in\sigma^i([b])$.\qedhere
    \end{proof}
    
    Now, we will construct a tree-shift $X$ that is not of finite type, but for any open set $U\subseteq X$ and $i\in\Sigma$, the set $\sigma^i(U)$ is open in $X$.\\
    
    \begin{example}
    We define the set $\mathcal{F}$ as follows. For $n>0$ and $b\colon\{0,1\}^{<n}\to\{0,1\}$, let $b\in\mathcal{F}$ if and only if, there are two different words $u,w\in\{0,1\}^{<n}$ of the same length, such that $b_w=b_u=0$. We put $X=\mathcal{T}_\mathcal{F}$. We defined $X$ by specifying its set of forbidden patterns, so $X$ is a tree-shift. A tree $t$ is in $X$ if and only if there are no two zeros in the same row of $t$.
    Clearly $X$ is a non-empty tree-shift on $\Sigma=\{0,1\}$ over $\mathcal{A}=\{0,1\}$.\\
    \indent Assume that $X$ is of finite type and its forbidden set $\mathcal{F}$ consists of blocks of height $n>0$. Then we can define $t\colon\{0,1\}^*\to\{0,1\}$ as follows. 
    We put $t_{0^{n+1}}=0$, $t_{1^{n+1}}=0$ and for words $w\in\Sigma^*\setminus\left\{0^{n+1},1^{n+1}\right\}$, we set $t_w=1$. Then, for any $w\in\Sigma^*$, at most one node labelled with 0 appears in $\restr{t}{w\Sigma^{<n}}$, so it is in $\mathcal{B}_n(X)$. This would mean that $t\in X$, but there are two words of the same length, such that there are zeros at their nodes, so $t$ can not be in $X$. This is a contradiction. Therefore $X$ is not of finite type.\\
    \indent We prove now that for every $i\in\Sigma$, the map $\restr{\sigma^i}{X}$ is open.
    Fix $i\in\Sigma$. We need to show that for $U=[b]\subseteq X$, where $b\in\mathcal{B}(X)$, the set $\sigma^i([b])$ is open in $X$. Let $b\in\mathcal{B}_n(X)$ for $n>0$ and $t\in\sigma^i([b])$. 
    Note that $t\in\sigma^i\left([b]\right)$ if and only if $t = \restr{\tilde{t}}{i\Sigma^*}$ for some $\tilde{t}\in [b]$.
    We will prove that $$\left[\restr{t}{\Sigma^{<n}}\right]\subseteq\sigma^i([b]).$$
    Let $s\in\left[\restr{t}{\Sigma^{<n}}\right]$. We will define $r\colon\{0,1\}^*\to\{0,1\}$ as follows. For $w\in\Sigma^*$, let $r_{iw}=s_w$. For $w\in\Sigma^{<n}$ that does not start with $i$, let $r_w=b_w$ and for all words $w\in\Sigma^*\setminus\Sigma^{<n}$ that do not start with $i$, let $r_w=1$. 
    The tree $r$ constructed in this way is an element of $[b]$.
    The first observation is that $\restr{r}{\Sigma^{<n}} = b$ as if $w\in\Sigma^{<n}$ does not start with $i$ then by definition $r_w = b_w$ and if $w = i\tilde{w}$ for some $\tilde{w}\in\Sigma^{<n-1}$ then 
    $$ r_w = s_{\tilde{w}} = t_{\tilde{w}} = \tilde{t}_{i\tilde{w}} = \tilde{t}_w = b_w.$$
    Also, $r\in X$ as it can be proved by considering a few cases. If $m\leq n$ then $r_u=b_u$ and $r_v = b_v$ so they can not be both 0, because $b\in\mathcal{B}(X)$.
    If $m>n$ and words $u,v$ both start with $i$ then they can be expressed as $u = i\tilde{u}, v = i\tilde{v}$ for some $\tilde{u},\tilde{v}\in\Sigma^{m-1}$.
    Then by the definition of $r$ we have $r_{u} = s_{\tilde{u}}$ and $r_{v} = s_{\tilde{v}}$.
    As $s\in X$, the labels of $s$ at $\tilde{u}$ and $\tilde{v}$ can not be both 0 and we are done with this case.
    If $m>n$ and one of words $u,v$ does not start with $i$ then this word is 1.
    Therefore $r$ is an element of $[b]$ and of course we have $s=\sigma^i(r)$ so $s\in\sigma^i([b])$.
    \end{example}

\section{Conclusion}
    We proved that for a tree-shift being of finite type is equivalent to the pseudo-orbit-tracing property and implies topological stability. To obtain the pseudo-orbit-tracing property from topological stability, we assumed that the considered tree-shift is perfect (has no isolated points).
    To the best of the author's knowledge, it is unknown if the equivalence between being of finite type and topological stability also holds without assuming that the shift space is perfect.
    Additionally, we proved that if $X$ is a tree-shift of finite type then the restricted shift maps $\left\{\restr{\sigma^i}{X}\right\}_{i\in\Sigma}$ are open, but the converse fails when we consider tree-shifts with alphabet $\Sigma$ consisting of more than one symbol. Therefore the relations between the notions can be summarized by the following diagram (here TS stands for "topologically stable, POTP for "pseudo-orbit-tracing property" and FT for "of finite type")
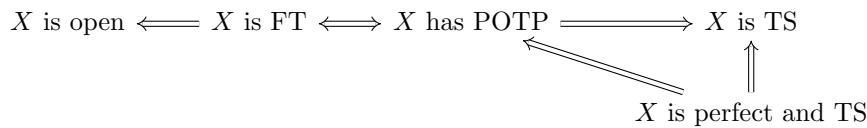
    \[
  \begin{tikzcd}
     X\text{ is open} \arrow[Leftarrow]{r}{} & X\text{ is FT} \arrow[Leftrightarrow]{r}{} & X\text{ has POTP} \arrow[Rightarrow]{r}{} & X\text{ is TS} \\
     & & & X\text{ is perfect and TS} \arrow[Rightarrow]{ul}{} \arrow[Rightarrow]{u}{}
  \end{tikzcd}
\]

\bibliographystyle{amsplain}
\bibliography{bibliography}

\providecommand{\bysame}{\leavevmode\hbox to3em{\hrulefill}\thinspace}
\providecommand{\MR}{\relax\ifhmode\unskip\space\fi MR }
\providecommand{\MRhref}[2]{%
  \href{http://www.ams.org/mathscinet-getitem?mr=#1}{#2}
}
\providecommand{\href}[2]{#2}
\begin{thebibliography}{10}

\bibitem{beal_1}
Nathalie Aubrun and Marie-Pierre Béal, \emph{Decidability of conjugacy of tree-shifts of finite type}, Automata, Languages and Programming Part 1 (2009), 132--143.

\bibitem{beal_2}
\bysame, \emph{Sofic and almost of finite type tree-shifts}, Computer Science -- Theory and Applications (2010), 12--24.

\bibitem{beal_3}
\bysame, \emph{Tree-shifts of finite type}, Theoret. Comput. Sci. \textbf{459} (2012), 16--25.

\bibitem{beal_4}
\bysame, \emph{Sofic tree-shifts}, Theory Comput. Syst. \textbf{53} (2013), 621--644.

\bibitem{beal_5}
\bysame, \emph{Tree algebra of sofic tree languages}, RAIRO Theor. Inform. Appl. \textbf{48} (2014), 431--451.

\bibitem{Ban:mixing}
Jung-Chao Ban and Chih-Hung Chang, \emph{Mixing properties of tree-shifts}, Journal of Mathematical Physics \textbf{58} (2017), no.~11, 112702.

\bibitem{Ban:mix_chaos}
\bysame, \emph{Tree-shifts: irreducibility, mixing, and the chaos of tree-shifts}, Transactions of the American Mathematical Society \textbf{369} (2017), no.~12, 8389--8407.

\bibitem{Bowen}
Rufus Bowen, \emph{On axiom a diffeomorphisms}, vol.~35, The Conference Board of the Mathematical Sciences by the American mathematical Society, 1978.

\bibitem{stability}
Noriaki Kawaguchi, \emph{Topological stability and shadowing of zero-dimensional dynamical systems}, Discrete \& Continuous Dynamical Systems \textbf{39} (2019), no.~5, 2743--2761.

\bibitem{Kulczycki_Kwietniak}
Marcin Kulczycki and Dominik Kwietniak, \emph{Shadowing vs. distality for actions of $\mathbb{R}^n$}, Dynamical Systems \textbf{27} (2012), no.~2, 205--211.

\bibitem{lind_marcus}
Douglas Lind and Brian Marcus, \emph{An introduction to symbolic dynamics and coding},  (1995), 1--63.

\bibitem{Oprocha2008ShadowingIM}
Piotr Oprocha, \emph{Shadowing in multi-dimensional shift spaces}, Colloquium Mathematicum \textbf{110} (2008), 451--460.

\bibitem{Osipov_Tikhomirov}
Alexey Osipov and Sergey Tikhomirov, \emph{Shadowing for actions of some finitely generated groups}, Dynamical Systems \textbf{29} (2013).

\bibitem{parry}
William Parry, \emph{Symbolic dynamics and transformations of the unit interval}, Trans. Amer. Math. Soc. \textbf{122} (1966), 368--378.

\bibitem{Walters1978}
Peter Walters, \emph{On the pseudo orbit tracing property and its relationship to stability}, pp.~231--244, Springer Berlin Heidelberg, Berlin, Heidelberg, 1978.

\end{thebibliography}

\end{document}